\documentclass[a4paper, 11pt]{article}

\usepackage[T1]{fontenc}
\usepackage{amsmath}
\usepackage{amssymb}
\usepackage{amsfonts}
\usepackage{amsthm}
\usepackage{indentfirst}
\usepackage{psfrag}
\usepackage{amscd,stmaryrd,latexsym}
\usepackage[pdftex]{graphicx} 
\newtheorem{thm}{Theorem}[section]
\newtheorem{lem}{Lemma}[section]

\newtheorem{Th}{Theorem}
\newtheorem{Prop}{Proposition}

\theoremstyle{remark}

\theoremstyle{definition}

\theoremstyle{remark}
\newtheorem{oss}{Remark}[section]

\newcommand{\be}{\begin{equation}}
\newcommand{\ee}{\end{equation}}
\newcommand{\R}{\mathbb{R}}

\newcommand{\N}{\mathbb{N}}

\newcommand{\C}{\mathbb{C}}
\newcommand{\CP}{\mathbb{C}\mathbb{P}}

\newcommand{\G}{\mathbf{g}}

\newcommand\res{\mathop{\hbox{\vrule height 7pt width .5pt depth 0pt
\vrule height .5pt width 6pt depth 0pt}}\nolimits}

\def\eps{\mathop{\varepsilon}}
\def\Mc{\mathop{\mathcal{M}_{}}}
\def\Tc{\mathop{\mathcal{T}_{}}}
\def\Gc{\mathop{\mathcal{G}}}
\def\Cc{\mathop{\mathcal{C}}}

\def\Ac{\mathop{\mathcal{A}_{}}}
\def\Ace{\mathop{\mathcal{A}_{\epsilon}}}
\def\Acer{\mathop{\mathcal{A}_{\epsilon, \rho}}}
\def\Acr{\mathop{\mathcal{A}_{\rho}}}
\def\Vc{\mathop{\mathcal{V}_{}}}
\def\Vce{\mathop{\mathcal{V}_{\epsilon}}}

\def\pt{\mathop{\left( \Phi^{-1} \right)}}
\def\Sc{\mathop{\mathcal{S}_{}}}
\def\Scr{\mathop{\mathcal{S}_{\rho}}}

\def\Sce{\mathop{\mathcal{S}_{\varepsilon}}}
\def\Scer{\mathop{\mathcal{S}_{\varepsilon, \rho}}}

\def\Hc{\mathop{\mathcal{H}_{}}}
\def\La{\Lambda}

\def\la{\lambda}

\def\Om{\Omega}
\def\om{\omega}
\def\P{\Phi}

\def\p{\partial}

\def\vt{\vartheta}
\def\Th{\Theta}

\begin{document}

\title{\textbf{Uniqueness of Tangent Cones to Positive-$(p,p)$ Integral Cycles}}
\author{\textit{Costante Bellettini\footnote{Princeton University \textit{Address}: Fine Hall, Washington Road
Princeton NJ 08544 USA \textit{E-mail address}: cbellett@math.princeton.edu}} } 
\date{}
\maketitle

\textbf{Abstract}: \textit{Let $(\Mc,\om)$ be a symplectic manifold, endowed with a compatible almost complex structure $J$ and the associated metric $g$. For any $p \in \{1, 2, ... \frac{1}{2}\text{dim}\Mc\}$ the form $\Om:=\frac{1}{p !}\om^p$ is a calibration. More generally, dropping the closedness assumption on $\om$, we get an almost hermitian manifold $(\Mc,\om, J, g)$ and then $\Om$ is a so-called semi-calibration. We prove that integral cycles of dimension $2p$ (semi-)calibrated by $\Om$ possess at every point a unique tangent cone. The argument relies on an algebraic blow up perturbed in order to face the analysis issues of this problem in the almost complex setting.}

\section{Introduction}
\label{setting}

The notion of calibration appeared in the foundational paper \cite{HL} in 1982, after that key features of calibrations had been observed in some particular cases in the previous decades (see \cite{Morgan} for a historical overview).

An immediate impact of calibrated currents was in connection with Plateu's problem, since these objects are mass-minimizers in their homology class and thus provide plenty of interesting and explicit examples of volume-minimizers. In the last fifteen years, however, calibrations have appeared surprisingly in many other geometric or physical problems, for example (see \cite{DoT}, \cite{SYZ}, \cite{Ta}, \cite{Ti}, \cite{Ti2}) theory of invariants, Yang-Mills fields, String theory, etc. Typically an essential issue in these studies is to understand regularity properties of calibrated currents.

\medskip

Already raised in \cite{HL}, one of the long-standing regularity questions is whether calibrated integral currents admit unique tangent cones. The issue is still open, except for currents of dimension $2$. Let us recall a few notions and the state of the art, before passing to the results of the present work.

\medskip

Given a $m$-form $\phi$ on a Riemannian manifold $(M,g)$, the comass of $\phi$ is defined to be
\[||\phi||^*:=  \sup \{\langle \phi_x, \xi_x \rangle: x \in M, \xi_x \text{ is a unit simple $m$-vector at } x\}.\]
A form $\phi$ of comass one is called a \textit{calibration} if it is closed ($d \phi = 0$). We will be dealing also with non-closed forms of unit comass, which will be referred to as \textit{semi-calibrations}, following the terminology in \cite{PR}. 

Let $\phi$ be a (semi)calibration; among the oriented $m$-dimensional planes of the Grassmannians $G(m, T_xM)$, we pick those on which $\phi$ agrees with the $m$-dimensional volume form. Representing oriented $m$-dimensional planes as unit simple $m$-vectors, we are thus selecting the subfamily of the so-called \textit{$m$-planes calibrated by $\phi$}:
\[\mathcal{G}(\phi) := \cup_{x \in M} \{\xi_x \in G(m, T_xM): \langle \phi_x, \xi_x \rangle = 1 \}.\]

\medskip

An integral current of dimension $m$ is said to be $\phi$-(semi)calibrated if, ${\Hc}^m$-almost everywhere, its approximate tangent planes belong to $\Gc$. A simple argument (see \cite{HL}) then shows that calibrated currents are homologically mass-minimizing, while semi-calibrated ones are almost minimizers (or $\la$-minimizers, using the terminology of \cite{DuzS}). 

\medskip

Examples of well-known calibrations are the symplectic form $\om$ in an almost K\"ahler manifold, its normalized powers $\frac{1}{p !}\om^p$, the Special Lagrangian calibration in Calabi-Yau $m$-folds, the Associative calibration, and many others. 

If we drop the closedness assumption on $\om$ in the definition of almost K\"ahler manifold, we get what is usually called an almost Hermitian manifold, and $\om$ and $\frac{1}{p !}\om^p$ are then semicalibrations. We will refer to these as almost-complex semi-calibrations.

\medskip

When dealing with a boundaryless integral current $C$, also called integral cycle, or simply when we localize the current to an open set in which the boundary is zero, it turns out that calibrated currents satisfy an important \textit{monotonicity} formula for the mass ratio: for any $x_0$, the quantity $\frac{M(C \res B_r(x_0))}{r^m}$ is a weakly increasing function of $r$. This is a classical result for mass-minimizers (see \cite{F}, \cite{SimonNotes} or \cite{Morgan}), proved for constant calibrations in \cite{HL}. 

When we turn our attention to almost-minimizers, what we get is an \textit{almost-monotonicity} formula, see \cite{DuzS} and \cite{SimonNotes}. Almost-monotonicity was proved for $C^1$ semi-calibrations in \cite{PR}: it states that the mass ratio $\frac{M(C \res B_r(x_0))}{r^m}$ is given by a weakly increasing function of $r$ plus a perturbation term, that is infinitesimal of $r$. The perturbation term is bounded in modulus independently of $x_0$.

\medskip

Immediate consequences of (almost) monotonicity are:

\begin{description}
 \item[(i)] the \textit{density} of the current is well-defined for every point $x_0$ as the limit 

$$\nu(x_0) := \lim_{r \to 0} \frac{M(C \res B_r(x_0))}{\alpha_m r^{m}} ,$$

where $\alpha_m$ is the $m$-dimensional volume\footnote{Recall that an arbitrary integral current $C$ is defined by assigning on an oriented $m$-rectifiable set $\Cc$ an integer-valued \textit{multiplicity function} $\theta \in L^1(\Cc, \N)$. For an arbitrary integral current, the density $\nu$ is well-defined ${\Hc}^m$-a.e. and agrees ${\Hc}^m$-a.e. with $\theta$. What we get for semicalibrated cycles is that the density $\nu$ is well-defined \underline{everywhere}, and we can take $\nu$ as the ``precise representative'' for the multiplicity $\theta$.} of the unit ball $B^m$.

 \item[(ii)] the density is an \textit{upper semi-continuous} function.
 
 \item[(iii)] the density of a semi-calibrated integral cycles is, everywhere on the support\footnote{The support of a current $C$ is the complement of the largest open set in which the action of the current is zero.}, bounded by $1$ from below \footnote{This fails for arbitrary integral currents, as the example of the current of integration on a cone (counted with multiplicity $1$) shows: the density at the vertex, although well-defined, depends on the opening angle of the cone (the narrower the cone is, the lower the density is). If we take, instead of a cone, a surface with a cusp point, the density there is $0$.}.
 
\end{description}

\medskip

Monotonicity further yields the existence of \textit{tangent cones}: this is a first step in the study of regularity of calibrated currents. The notion of tangent cone to a current $C$ at a point $x_0$ is defined by the following procedure, called the \textit{blow up limit}, whose idea goes back to De Giorgi \cite{DG}. Dilate $C$ around $x_0$ of a factor $r$; in normal coordinates around $x_0$ this amounts to pushing forward $C$ via the map  $\displaystyle \frac{x-x_0}{r}$:

\be
\label{eq:defblowup}
(C_{x_0,r} \res B_1 )(\psi):=\left[ \left( \frac{x-x_0}{r}\right)_\ast C \right](\chi_{B_1} \psi) = C \left( \chi_{B_r(x_0)}\left(\frac{x-x_0}{r}\right)^\ast \psi\right).
\ee

The fact that $\displaystyle \frac{M(C \res B_r(x_0))}{r^{2}}$ is monotonically almost-decreasing as $r \downarrow 0$ gives that, for $r\leq r_0$ (for a small enough $r_0$), we are dealing with a family of cycles $\{C_{x_0,r} \res B_1\}$ in $B_1$ that are equibounded in mass. Therefore Federer-Fleming's compactness theorem (see e.g. \cite{G} page 141) gives that there exist weak limits of $C_{x_0,r}$ as $r \to 0$. Every such limit $C_\infty$ is an integer multiplicity rectifiable boundaryless current which turns out to be a cone\footnote{A current is said to be a cone with vertex $p$ if it is invariant under omotheties centered at $p$.} calibrated by $\om_{x_0}$ and is called a tangent cone to $C$ at $x_0$. The density of each tangent cone at the vertex is the same as the density of $C$ at $x_0$ (see \cite{HL}). 

\medskip

The natural first question, raised already in \cite{HL}, is whether Federer-Fleming's compactness theorem can yield different sequences of radii with different cones as limits, i.e. whether the tangent cone at an arbitrary point is \textit{unique} or not. The answer is positive for semi-calibrated integral cycles of dimension $2$, as proved in \cite{PR}. The uniqueness is also known for mass-minimizing integral currents of dimension $2$, thanks to \cite{W}. In some other cases, which also follow from either of the forementioned \cite{W} or \cite{PR}, the proof has been achieved using techniques of positive intersection, namely for integral pseudo-holomorphic cycles in dimension $4$ (\cite{Ta}, \cite{RT1}) and for integral Special Legendrian cycles in dimension $5$ (\cite{BR}, \cite{B}). In \cite{RT2} the uniqueness for pseudo holomorphic integral $2$-dimensional cycles is achieved in arbitrary codimension. In \cite{Simon} it is proved that if a tangent cone to a minimal integral current has multiplicity one and has an isolated singularity, then it is unique. 

In dimensions higher than two, apart from L. Simon's result \cite{Simon}, the uniqueness of tangent cones is a completely open question. In this work we give a positive answer in the case of pseudo-holomorphic integral currents (i.e. semi-calibrated by the almost complex semi-calibration $\frac{1}{p !}\om^p$) of arbitrary dimension and codimension.

\medskip

The uniqueness cannot be obtained merely as a consequence of the monotonicity of the mass ratio. Indeed, the notion of being calibrated by $\om$ can be extended from integral currents to normal ones (then it is usually called \textit{$\om$-positiveness}, as in \cite{HL}). Normal positive cycles still fulfil the same monotonicity formula, but it was proved in \cite{Kis} that they might have non-unique tangent cones.

\medskip

The approach presented in this work relies on an algebraic blow up technique, adapted to the almost complex setting, that shows how, for almost-complex semi-calibration, uniqueness of tangent cones can indeed be obtained for \underline{integral} semi-calibrated cycles just as a consequence of almost monotonicity and of the fact that the density is bounded by $1$ from below. 

The same approach was used in \cite{B2} for pseudo-holomorphic non-rectifiable currents of dimension $2$. Here we extend that technique to higher dimensional pseudo-holomorphic currents: these are also referred to as \textit{positive-$(p,p)$ currents}, as we will see in the next section, where we describe more closely the setting and the result.

\medskip

\section{The main result}
\label{result}

Let $(\mathcal{M}, J)$ be an almost complex manifold of dimension $2n+2$, where $J$ is an almost complex structure. Let $\om$ be a non-degenerate differential form of degree $2$ compatible with $J$. If $d \om = 0$ then we have a symplectic form, but we will not need to assume closedness. Let $g$ be the Riemannian metric defined by $g(\cdot, \cdot) := \om(\cdot, J \cdot)$. The triple $(\mathcal{M}, J, g)$ is an almost Hermitian manifold; when the associated form $\om$ is closed, we get an almost K\"ahler manifold. The word ``almost'' refers to the fact that $J$ can be non-integrable.

\medskip

On the other hand, given a constant non-degenerate differential $2$-form $\om$ in a $(2n+2)$-dimensional vector space $V$ with a scalar product, if $\om$ has unit comass then the non-zero form $\om^{n+1}$ is the volume form and $(V, \om, g)$ is a symplectic vector space with a compatible metric. So given any Riemannian manifold $(\mathcal{M}, g)$ (where $g$ is the metric), if $\om$ is a non-degenerate two-form with pointwise unit comass, then the non-zero form $\om^{n+1}$ is the volume form and, for any $x \in \Mc$, $(T_x \Mc, \om, g)$ is a symplectic vector space with a compatible metric. Then we can define an almost complex structure $J$ on the tangent bundle of $\Mc$ by setting $J:=g^{-1} \om$.  

\medskip

It is therefore equivalent for the purposes of this work to start with an almost complex structure $J$ or with a semicalibration $\om$ having pointwise unit comass. We can always assume to have a compatible triple $(\om, g, J)$ on $\Mc$.

\medskip

For any fixed integer $p \in \{1, 2, ... n\}$ the form $\Om:=\frac{1}{p !}\om^p$ is a semi-calibration on $\Mc$ for the metric $g$, i.e. the comass $\|\Om\|^*$ is $1$. This is nothing else but Wirtinger's inequality (\cite{Wirtinger}).	
If $\om$ is closed, then so is $\Om$ and we get a calibration.

\medskip

Recall that the family $\mathcal{G}(\Om)$ of \textit{$2p$-planes calibrated by $\Om$} is
\[\mathcal{G}(\Om) := \cup_{x \in \Mc} \; {\Gc}_x := \cup_{x \in \Mc} \{\xi_x \in G_{2p}(x, T_x \Mc): \langle \Om_x, \xi_x \rangle = 1 \}.\]

By Wirtinger's theorem \cite{Wirtinger} ${\Gc}_x$ is made exactly of the $2p$-dimensional $J_{x}$-complex subspaces of $T_x \Mc$. For this reason the $2p$-planes in $\mathcal{G}(\Om)$ are also called \textbf{positive-$(p,p)$} vectors. It is worthwile noticing that the property of being a complex subspace of $(T_x \Mc, J_x)$ is not affected by choosing different couples $(g_x, \om_x)$ and $(g'_x, \om'_x)$ compatible with $J_x$ in $T_x \Mc$. Therefore a positive-$(p,p)$ vector is calibrated both by $\Om:=\frac{1}{p !}\om^p$ in $(T_x \Mc, g)$ and $\Om':=\frac{1}{p !}\om'^p$ in $(T_x \Mc, g')$. This observation is of key importance for our proof.

\medskip

We will consider a $2p$-dimensional integral cycle $T$ (semi)-calibrated by $\Om$. This means that $\vec{T}$ (the oriented approximate tangents to $T$) belong ${\Hc}^{2p}$-a.e. to $\Gc(\Om)$, or equivalently that $T(\Om) = \int_{\Mc} \langle \Om, \vec{T} \rangle d\|T\| = M(T)$.

The absence of boundary means that it holds, for any compactly supported $(2p-1)$-form $\alpha$, $(\p C)(\alpha) := C(d \alpha) = 0$.

\medskip

The issues we will be dealing with, namely tangent cones, are local: we will be only interested in the asymptotic behaviour of currents around a point, so we can assume to work in a chart rather than on a manifold. $\Om$-positive normal cycles in $\R^{2n+2}$ satisfy the following important \textit{almost monotonicity property} for the mass-ratio at any point $x_0$.

\begin{Prop}
\label{Prop:monotonicity}
Let $\R^{2n+2}$ be endowed with a Riemannian metric $g$ and a non-degenerate two-form $\om$ of unit comass. Denote by $\Om$ the semi-calibration $\frac{1}{p !}\om^p$. Let the $2p$-dimensional normal cycle $T$ be $\Om$-positive and let $x_0$ be an arbitrary point. Denote by $B_r(x_0)$ the geodesic ball around $x_0$ of radius $r$.

For an arbitrarily chosen point $x_0$, the mass ratio $\frac{M(T \res B_r(x_0))}{r^{2p}}$ is an almost-increasing function in $r$, i.e. $\displaystyle \frac{M(T \res B_r(x_0))}{r^{2p}}=R(r) + O(r)$ for a function $R$ which is monotonically non-increasing as $r \downarrow 0$ and a function $O(r)$ which is infinitesimal. 
\end{Prop}

This is proved in \cite{PR}, Proposition 1. It is important to notice that the same proof works if we assume the semi-calibration to be just Lipschitz continuous rather than $C^1$, see the appendix of \cite{B2}. The perturbation term $O(r)$ is bounded, \underline{independently} of $x_0$, by $C \cdot L \cdot r$, where $C$ is a dimensional constant and $L$ is the Lipschitz constant of the semi-calibration.

\medskip

\medskip

In this work we prove:

\begin{thm}
\label{thm:mainrect}
Given an almost complex $(2n+2)$-dimensional manifold $(\mathcal{M}, J, \om, g)$, with a non-degenerate two-form $\om$ compatible with $J$ and associated Riemannian metric $g$, denote by $\Om$ the form $\Om:=\frac{1}{p !}\om^p$, for a fixed $p \in \{1, 2, ..., n\}$; let $T$ be a positive-$(p,p)$ integral cycle, i.e. an integral $2p$-cycle semi-calibrated by $\Om$.

Then for any $x_0$ the tangent cone to $T$ at $x_0$ is unique.
\end{thm}

\medskip

Theorem \ref{thm:mainrect} reduces to the result in \cite{PR} in the case $p=1$ and follows from \cite{W} if $p=1$ and $d \om=0$.

\medskip

With a suitable choice of coordinates we can identify the tangent space $T_{x_0} \mathcal{M}$, endowed with the complex structure $J_{x_0}$, with $\C^{n+1}$: then every tangent cone $T_\infty$ to $T$ at $x_0$ is a positive-$(p,p)$ cone in $\C^{n+1}$: such a cone is uniquely defined by a holomorphic $(p-1,p-1)$ integral cycle $L_\infty$ in $\CP^n$. 

Using the regularity theory for holomorphic integral cycles (\cite{K}, \cite{HS}, \cite{Ale}) we can deduce that $L_\infty$ is in fact the sum of a finite number of holomorphic algebraic varieties\footnote{We are slighlty abusing language here: these are algebraic varieties that are holomorphic away from their possible singular set.}, each one taken with a constant integer multiplicity, but we will not need this result.

\medskip

We will prove first the following

\begin{lem}
\label{lem:uniqsupport}
Let $T$ be a $\Om$-semi-calibrated integral cycle and $x_0$ an arbitrary point. Then all tangent cones to $T$ at $x_0$ have a uniquely determined support.
\end{lem}

Once this lemma is achieved, the uniqueness of tangent cones (i.e. theorem \ref{thm:mainrect}) follows with a few extra considerations (without making use of the results in \cite{K}, \cite{HS}, \cite{Ale}) developed in section \ref{proof}, namely: \textbf{(i)} the space of tangent cones to $T$ at $x_0$ is closed and connected in the space of $2p$-integral cycles, \textbf{(ii)} the density is continuous under convergence of calibrated integral cycles sharing the same support. 

\medskip

As for lemma \ref{lem:uniqsupport}, the key idea for its proof is the analysis implementation, in the almost complex setting in which we are working, of the classical algebraic blow up. This was already used in \cite{B2} for positive-$(1,1)$ normal cycles and is here generalized to higher dimensional pseudo-holomorphic currents. The technique clearly shows that the uniqueness of theorem \ref{thm:mainrect} holds just for ``density reasons'' (recall that the uniqueness can fail when we look at non-rectifiable currents, where the density is allowed to take any values $\geq 0$, see \cite{Kis} and \cite{B2}).

\section{Strategy and tools for the proof of the theorem}
\label{tools}

The first important remarks are contained in the following

\begin{lem}
\label{lem:ovvio}
Let $T$ be as in theorem \ref{thm:mainrect} and be $\Tc$ the support of $T$. Assume that there exists a sequence of points $x_m \in \Tc$ with $x_m \to x_0$ and $x_m \neq x_0$ such that $\frac{x_m - x_0}{|x_m - x_0|} \to y \in S^{2n+1}$. Then there exists a tangent cone to $T$ at $x_0$, say $T_\infty$, such that the point $y$ belongs to the support of $T_\infty$.

On the other hand, if $y \in S^{2n+1}$ belongs to the support of a tangent cone $T_\infty$ to $T$ at $x_0$, then there exists a sequence $x_m \to x_0$ (with $x_m \neq x_0$) of points $x_m$ in the support of $T$ such that $\frac{x_m - x_0}{|x_m - x_0|} \to y$.
\end{lem}

\begin{oss}
\label{oss:contmass}
Let $C_k \rightharpoonup C_\infty$ be a sequence of $\phi_k$-semi-calibrated integral cycles ($k \in \N \cup\{\infty\}$), where $\phi_k$ are semi-calibrations with respect to the metrics $g_k$, and assume that the $\phi_k$ converge uniformly to $\phi_\infty$, $g_k$ converge uniformly to $g_\infty$ and the $C_k$'s have equibounded masses. Then $M(C_k \res B) \to M(C_\infty \res B)$ for any open set $B$. This follows because computing the mass for a semicalibrated current amounts to testing the current on the semi-calibration, so the convergence of the masses follows from the definition of weak-convergence of currents.
\end{oss}

\begin{oss}
Recall that, as a consequence of monotonicity, a point belongs to the support of a semi-calibrated integral cycle if and only if its density is $\geq 1$.
\end{oss}

\begin{proof}[\textbf{proof of lemma \ref{lem:ovvio}}]

The first statement follows by choosing the sequence of radii $r_m:=|x_m - x_0|$ and by looking at the sequence $T_{x_0, r_m}$. Up to a subsequence we may assume that $T_{x_0, r_m} \rightharpoonup T_\infty$. Each $x_m$ is of density $\geq 1$ for $T$ by assumption and, for any $m$, the point $\frac{x_m - x_0}{|x_m - x_0|}$ is of density $\geq 1$ for $T_{x_0, r_m}$. Since $\frac{x_m - x_0}{|x_m - x_0|} \to y$, analogously to remark \ref{oss:contmass} we can get $$M\left(T_{x_0, r_m} \res B_R\left(\frac{x_m - x_0}{|x_m - x_0|}\right)\right) \to M(T_\infty \res B_R(y))$$ for any $R>0$. By the almost monotonicity formula $M(T_\infty \res B_R(y)) \geq \alpha_{2p}R^{2p}$ and so $y$ is a point of density $\geq 1$ for $T_\infty$.

\medskip

Let now $y \in S^{2n+1}$. If there exists no sequence $x_m \neq x_0$ such that $x_m \in \Tc$, $x_m \to x_0$ and $\frac{x_m - x_0}{|x_m - x_0|} \to y$, then we can assume to have a ball $B^{2n+1}_a(y) \subset S^{2n+1}$ such that the cone $0 \sharp B^{2n+1}_a(y)$ is disjoint from $\Tc \cap B^{2n+2}_R(0)$, for some small $R>0$. But then, for any dilation $T_{x_0,r}$ with $r<R$ we have $M(T_{x_0,r} \res B_a^{2n+2}(y))=0$. Since the mass passes to the limit for convergence of semi-calibrated cycles (remark \ref{oss:contmass}), we  deduce that $y$ is a point of density $0$ for any limit of the family $T_{x_0,r}$, therefore it cannot appear as a point in the support of any tangent cone.
\end{proof}

In order to achieve lemma \ref{lem:uniqsupport}, it suffices, thanks to lemma \ref{lem:ovvio}, to analyze limits of $\frac{x_m - x_0}{|x_m - x_0|} \to y$ for $x_m \in \Tc$, $x_m \to x_0$. More precisely, recalling that each tangent cone is a holomorphic $(p,p)$-cone, if $y \in S^{2n+1}$ belongs to the support of a tangent cone $T_\infty$, then every point in the Hopf fiber $\{e^{i \theta} y\}_{\theta \in [0, 2 \pi)}$ is also a point whose density for $T_\infty$ equals that of $y$. In other words, if $y$ is in the support of $T_\infty$, so is the whole fiber $\{e^{i \theta} y\}_{\theta \in [0, 2 \pi)}$. Denote by $H:S^{2n+1} \to \CP^n$ the standard Hopf projection. Then, in order to prove lemma \ref{lem:uniqsupport}, we actually need to show the following

\begin{Prop}
\label{Prop:restate}
Let $T$ be a positive-$(p,p)$ integral cycle. Let $\{x_m\}$ be a sequence of points such that $x_m \in \Tc$ with $x_m \to x_0$, $x_m \neq x_0$ and $H\left(\frac{x_m - x_0}{|x_m - x_0|}\right) \to y \in \CP^n$. Then the support of any tangent cone to $T$ at $x_0$ must contain the Hopf circle $H^{-1}(y)$. 
\end{Prop}

This proposition will be proved by employing an algebraic blow up of the semi-calibrated current $T$ around $x_0$. We now shortly recall the notations and the construction, which is developed in more detail in \cite{B2}.

\medskip

Since tangent cones to $T$ at a point $x_0$ are a local issue, we can assume straight from the beginning to work in the unit geodesic ball, in normal coordinates centered at $x_0$; for this purpose it is enough to start with the current $T$ already dilated enough around $x_0$. Always up to a dilation, without loss of generality we can actually start with the following situation.

$T$ is a $\Om$-positive normal cycle in the ball $B^{2n+2}_2(0)$, the coordinates are normal with respect to the origin, $J$ is the standard complex structure at the origin, $\om$ is the standard symplectic form at the origin, $\|\om - \om_0\|_{C^{2,\nu}}(B^{2n+2}_2)$, $\|\Om - \Om_0\|_{C^{2,\nu}}(B^{2n+2}_2)$ and $\|J - J_0\|_{C^{2,\nu}}(B^{2n+2}_2)$ are small enough.

\medskip

\textbf{How to blow up the origin.}
We shall be using standard coordinates $(z_0, z_1, ..., z_n)$ in $B^{2n+2}_2(0) \subset \C^{n+1} \cong \R^{2n+2}$ and the following notations as in \cite{B2}: $$\Sce:=\{(z_0, z_1, ... z_n) \in B_{1+\eps}^{2n+2} \subset \C^{n+1}: |(z_1, ..., z_n)| < (1+\eps)|z_0|\},$$

$$\Vce \subset \CP^n, \,\, \Vce :=\{[z_0, z_1, ..., z_n]:  |(z_1, ..., z_n)| < (1+\eps)|z_0|\}.$$

$$\text{for $X=[Z_1, ... , Z_{n+1}] \in \Vce$ } D^X \text{ is the ``straight'' $2$-plane }$$

$$\text{made of all points } \{\zeta (Z_1, ... Z_{n+1}): \zeta \in \C \}.$$

We also write $\Sc$ for ${\Sc}_0$ and $\Vc$ for ${\Vc}_0$. 

As shown in section 3 of \cite{B2}, by constructing (via a fixed point theorem) a \textit{pseudo-holomorphic polar foliation} we can produce an appropriate diffeomorphism 

\begin{equation}
\label{eq:psi}
 \Psi: \Sce \rightarrow \Psi(\Sce) \approx \Sce,
\end{equation}

which is close to the identity on $\Sce$, and which (by pulling-back the problem via $\Psi$) allows us to make an extra assumption on the almost complex structure $J$: namely the ``straight $2$-planes'' $D^X$ are $J$-pseudo holomorphic for all $X \in \Vce$. Figure 1 in \cite{B2} visually explains the behaviour of $\Psi$. 

\medskip

With this extra assumption on $J$, we can proceed to blow-up the origin of $\C^{n+1}$ as follows.

\medskip

\textit{Reminder}: \textit{algebraic blow-up} (from \textit{symplectic} or {\it algebraic geometry}, see \cite{MS}). Define $\widetilde{\C}^{n+1}$ to be the submanifold of $\CP^n \times \C^{n+1}$ made of the pairs $(\ell, (z_0, ... z_n))$ such that $(z_0, ... z_n) \in \ell$. 

Denote by $I_0$ the complex structure that $\widetilde{\C}^{n+1}$ inherits from $\CP^n \times \C^{n+1}$. Let $\Phi:\widetilde{\C}^{n+1} \rightarrow \C^{n+1}$ be the projection map $(\ell, (z_0, ... z_n)) \rightarrow (z_0, ... z_n)$. This map is holomorphic for the standard complex structures $J_0$ on $\C^{n+1}$ and $I_0$ on $\widetilde{\C}^{n+1}$ and is a diffeomorphism between $\widetilde{\C}^{n+1} \setminus \left( \CP^n \times \{0\} \right)$ and $\C^{n+1} \setminus  \{0\}$. Moreover the inverse image of $\{0\}$ is $\CP^n \times \{0\}$.

We will endow $\widetilde{\C}^{n+1}$ with other almost complex structures, different from $I_0$, so $\widetilde{\C}^{n+1}$ should be thought of just as an oriented manifold and the structure on it will be specified in every instance. The transformation $\Phi^{-1}$ (called \textit{proper transform}) sends the point $0 \neq (z_0, ... z_n) \in \C^{n+1}$ to the point\- $([z_0, ... z_n], (z_0, ... z_n)) \in \widetilde{\C}^{n+1} \subset \CP^n \times \C^{n+1}$. 

We will keep using the same letters $\P$ and $\Phi^{-1}$ to denote the same maps restricted to 

$$\Sce \subset B^{2n+2} \subset \R^{2n+2} \cong \C^{n+1} \text{ and } \Ace := \pt (\Sce) \subset \widetilde{\C}^{n+1},$$

also when we look at these spaces just as oriented manifolds (not complex ones). We will make use of the notation

$$\Scer :=\Sce \cap B_\rho^{2n+2} \subset \R^{2n+2} \cong \C^{n+1} \text{ and } \Acer := {\P}^{-1} (\Scer) \subset \widetilde{\C}^{n+1}.$$

\medskip

Let $g_0$ denote the standard metric\footnote{The standard metric on $\CP^n \times \C^{n+1}$ is the product of the Fubini-Study metric on $\CP^n$ and the flat metric on $\C^{n+1}$.} on ${\Ace}$ as a subset of $\widetilde{\C}^{n+1} \subset \CP^n \times \C^{n+1}$ and ${\vt}_0$ be the standard symplectic form on $\Ace$, uniquely defined by ${\vt}_0(\cdot, \cdot):= g_0(\cdot, -I_0 \cdot)$.

Define on ${\Ace} \,\setminus (\CP^n \times \{0\})$:

\begin{itemize}
 \item the almost complex structure $I := {\P}^* J$, $I(\cdot):= {\pt}_* J {\P}_*(\cdot)$,
 
 \item the metric $\G(\cdot, \cdot):= g_0(\cdot, \cdot) + g_0(I \cdot, I \cdot)$,
 
 \item the non-degenerate two-form $\vt(\cdot, \cdot):= \G(\cdot, -I \cdot)$.
\end{itemize}

The triple $(I, \G, \vt)$ makes ${\Ace} \setminus (\CP^n \times \{0\})$ an almost complex manifold and from \cite{B2} we have

\begin{lem}
\label{lem:lipcontrolI}
The triple $(I, \G, \vt)$ extended to ${\Ace}$ by setting it to be $(I_0, g_0, \vt_0)$ on $\Vce \times \{0\}$ is Lipschitz continuous on ${\Ace}$ and fulfils $$|I-I_0|(\cdot) \leq c \text{dist}_{g_0}(\cdot,\CP^n \times \{0\}),$$

$$|\G-g_0|(\cdot) \leq c \text{dist}_{g_0}(\cdot,\CP^n \times \{0\}),$$

$$|\vt-{\vt}_0|(\cdot) \leq c \text{dist}_{g_0}(\cdot,\CP^n \times \{0\}),$$ for some constant $c>0$, which is $o(1)$ of $|J-J_0|$. 
\end{lem}

Set $\Th:=\frac{1}{p !}\vt^p$ on $\Ace$. The aim is now to translate our original problem to the new space $({\Ac}, I, \G, \vt)$. For any $\rho>0$ we can take the proper transform of $T \res (\Sce \setminus \Scer )$, since $\pt$ is a diffeomorphism away from the origin:

$$P_{\rho}:=  {\pt}_* \left(T \res (\Sce \setminus \Scer )\right).$$

The current $P_\rho$ is clearly positive-$(p,p)$ in $({\Ac}, I, \G, \vt)$. What happens when $\rho \to 0$ ? Here is the answer.

\begin{lem}
\label{lem:normalposcycle}
\textbf{(i)} The current $P:=\lim_{\rho \to 0}  P_\rho = \lim_{\rho \to 0} {\pt}_* \left(T \res (\Sce \setminus \Scer)\right)$ is well-defined as the limit of currents of equibounded mass. The mass of $P$ (both with respect to $\G$ and to $g_0$) is bounded by a dimensional constant $C$ times the mass of $T$.

\textbf{(ii)} $P$ it is an integral \underline{cycle} in $\Ac$ and it is semi-calibrated by $\Th$. 
\end{lem}

\medskip

The proof is analog to the one in \cite{B2}, but we need to take care of the fact that the dimension of the current is higher. With the same notations as in \cite{B2}, for any $\rho$ consider the dilation $\lambda_\rho(\cdot):=\frac{\cdot}{\rho}$, sending $B_\rho$ to $B_1$, and the map 

\be
\label{eq:dilblow}
\Lambda_\rho:\Acr \to \Ac, \;\; \Lambda_\rho:=\Phi^{-1} \circ \lambda_\rho \circ \Phi ,
\ee

which in the coordinates of $\CP^n \times \C^{n+1}$ (the ambient space in which $\Ac$ is embedded) reads $\Lambda_\rho (\ell, z) = \left(\ell, \frac{z}{\rho}\right)$.

\begin{proof}[\textbf{proof of lemma \ref{lem:normalposcycle} (i)}]
Recall that we have a uniform bound $M(T_{0,r}) \leq K$, for a constant $K$ independent of $r$.

Each $P_\rho = {\pt}_* \left(T \res (\Sce \setminus \Scer)\right)$ is $\Th$-positive by construction, so $M(P_\rho)=P_\rho(\Th)$, where the mass is computed here with respect to $\G$. In order to study the limit as $\rho \to 0$, it is enough to look at an arbitrary sequence $\rho_n \to 0$ and prove that $P_{\rho_n}$ have equibounded masses and thus converge to a limit $P$, which must then be the limit of the whole family $P_\rho$.

\textit{1st step: choice of the sequence}. Denote by $\langle T, |z|=r\rangle$ the slice of $T$ with the sphere $\p B_r$. Choose $\rho_k$ so to ensure

\begin{itemize}
 \item (i) $T_{\rho_k} \rightharpoonup T_\infty$ in ${\Sc}$ for a certain cone $T_\infty$,
 
 \item (ii) $M(\langle T_{\rho_k}, |z|=1 \rangle)$ are equibounded by $4K$, 
 
 or, equivalently, $M \left( \langle T, |z|=\rho_k \rangle \right) \leq 4 K \rho_k^{2p-1}$.
\end{itemize}

This is just like step 1 of lemma 4.2 in \cite{B2}.

\medskip

\textit{2nd step: uniform bound on the masses}. We use in $\Ac$ standard coordinates inherited from $\CP^n \times \C^{n+1}$, i.e. we have $2n$ horizontal variables (from $\CP^n$) and $2n+2$ vertical variables. We want to estimate $M(P_\rho)=P_\rho(\Th) = P_\rho(\Th_0)+P_\rho(\Th-\Th_0)$, where $\Th_0:=\frac{1}{p !}\vt_0^p$ on $\Ac$. From lemma \ref{lem:lipcontrolI} we get that $|\Th-\Th_0|(p) < c \;\text{dist}_{g_0}(p,\CP^n \times \{0\})$ (we keep denoting the constant by $c$, although it is generally different that the one in lemma \ref{lem:lipcontrolI}; what is important is that it is still controlled by $|J-J_0|$).

\medskip

Recall that $\Th_0:=\frac{1}{p !}\vt_0^p$; the domain $\Sce$ is a product $\CP^n \times \C$ and therefore the standard form $\vt_0$ is $\vt_{\CP^n} + \vt_{\C^n}$, where $\vt_{\CP^n}$ is the standard symplectic form on $\CP^n$ extended to $\Sce$ (so independent of the two ``vertical variables'') and $\vt_{\C^n}$ is the symplectic two-form on $\C^n$, extended to $\Sce$ (so independent of the ``horizontal variables'').

Taking the $p$-th wedge power we get 

\be
\label{eq:wedgeproducts}
\Th_0:=\frac{1}{p !}\vt_0^p = \frac{1}{p !}(\vt_{\CP^n})^p + \sum_{m=1}^{p} (\vt_{\CP^n})^{p-m} (\vt_{\C})^m .
\ee

\medskip 

Let us first estimate $|P_\rho((\vt_{\CP^n})^p)|$. 

${\pt}^*(\vt_{\CP^n})=\p \overline{\p} \log \left( 1 + \sum_{j=1}^n \frac{|z_j|^2}{|z_0|^2} \right)$ in the domain $\Sce$ where $z_0 \neq 0$. In particular 

$${\pt}^*(\vt_{\CP^n}) = d \eta, \text{ where } \eta=\frac{1}{2}\left( \overline{\p} \log \left( 1 + \sum_{j=1}^n \frac{|z_j|^2}{|z_0|^2} \right) - \p \log \left( 1 + \sum_{j=1}^n \frac{|z_j|^2}{|z_0|^2} \right)\right).$$

We thus have 

$$P_\rho((\vt_{\CP^n})^p) =  \left(T \res (\Sce \setminus \Scer)\right)({\pt}^* (\vt_{\CP^n})^p) = \frac{1}{p !}\left(T \res (\Sce \setminus \Scer)\right)( d \eta)^p =$$

$$=\frac{1}{p !}\p\left[T \res (\Sce \setminus \Scer)\right] \left[\eta \wedge (d \eta)^{p-1}\right].$$

The boundary of $T \res (\Sce \setminus \Scer)$ is made of three portions: two live in the spheres $\p B_1$ and $\p B_\rho$ and the third one is given by the slice with a hypersurface of the form $\sum_{j=1}^n \frac{|z_j|^2}{|z_0|^2} = \text{const}$. The explicit form of $\eta$ then implies that this latter portion of boundary has zero action on $ \eta \wedge (d  \eta)^{p-1}$. We can thus write

$$P_\rho((\vt_{\CP^n})^p) =  \langle (T \res \Sce), r=1 \rangle  \left[\eta \wedge (d  \eta)^{p-1}\right] -  \langle (T \res \Sce), r=\rho \rangle \left[ \eta \wedge (d  \eta)^{p-1}\right].$$

Now observe the comass of $ \eta \wedge (d  \eta)^{p-1}$ on the spheres $\p B_\rho$. The comasses are equivalent up to a universal constant $C(p,n)$ to the maximum modulus of the coefficients.

For $\eta$ we can explicitly compute $|\eta| \leq \frac{K}{\rho}$ and $|d \eta| \leq \frac{K}{\rho^2}$ on $\p B_\rho$.

Now we focus on the sequence $\rho_k$ chosen in step 1, for which it holds $M\left \langle (T \res \Sce), r=\rho_k \rangle \right) \leq 4K \rho_k^{2p-1}$. We thus get $P_{\rho_k}(\Th_0) \leq K(p,n)$ independently of $\rho_k$.

\medskip

We pass now to estimating the other wedge products $\sum_{m=1}^{p} (\vt_{\CP^n})^{p-m} (\vt_{\C^n})^m$ left from (\ref{eq:wedgeproducts}); the key observation is that ${\pt}^*(\vt_{\C^n})$ has unit comass, and therefore the forms ${\pt}^*\left((\vt_{\CP^n})^{p-m} (\vt_{\C^n})^m\right)$ for $m \in \{1, ..., p\}$ all have comasses bounded by $\frac{K}{\rho^{2p-2}}$, where $\rho$ is the distance from the origin and $K$ is a universal constant. We then argue using a dyadic decomposition for the estimate on $|P_\rho(\Th-\Th_0)|$, as follows.

Break up $\Sce = \cup_{j=0}^\infty A_j$, where $A_j = \Sce \cap \left(B_{\frac{1}{2^j}} \setminus B_{\frac{1}{2^{j+1}}}\right)$. It holds $M(T \res A_j) \leq K \frac{1}{2^{2pj}}$. On the other hand, in the same annulus $A_j$, the comass of the form ${\pt}^*\left(\sum_{m=1}^{p} (\vt_{\CP^n})^{p-m} (\vt_{\C})^m\right)$ is $\leq K(p,n)\, 2^{2(p-1)(j+1)}$, for a constant $K(p,n)$ which only depends on the dimensions involved.

Therefore summing on all $j$'s we can bound 

$$\left|P_\rho\left(\sum_{m=1}^{p} (\vt_{\CP^n})^{p-m} (\vt_{\C})^m\right)\right|= \left|(T \res \Sce)\left({\pt}^* (\sum_{m=1}^{p} (\vt_{\CP^n})^{p-m} (\vt_{\C})^m)\right)\right|  \leq $$ $$\leq K(p,n) \sum_{j=0}^\infty  2^{2(p-1)(j+1)} \frac{1}{2^{2pj}} = K(p,n) \sum_{j=0}^\infty 2^{2p-2-2j} < \infty,$$

therefore $\left|P_\rho\left(\sum_{m=1}^{p} \vt_{\CP^n})^{p-m} (\vt_{\C^n})^m\right)\right|$ is also equibounded independently of $\rho$.

\medskip

To conclude the proof of part (i) of lemma \ref{lem:normalposcycle}, we must still prove that $|P_\rho(\Th-\Th_0)|$ is finite. Thanks to the Lipschitz control on $\vt-\vt_0$, which also yields $|\Th-\Th_0|(\cdot) \leq   c \text{dist}_{g_0}(\cdot,\CP^n \times \{0\})$, the form ${\pt}^* (\Th-\Th_0)$ in $\Sce$ has comass $\leq \frac{K}{\rho^{2p-1}}$, where $\rho$ is the distance from the origin. Arguing with a dyadic decomposition as done above, we find that also $|P_\rho(\Th-\Th_0)|$ is bounded independently of $\rho$.

\medskip

We have thus obtained that $M(P_{\rho_n})$ are uniformly bounded as $\rho_n \to 0$ and therefore there exists a current $P$ in $\Ace$ such that $P_{\rho} \rightharpoonup P$.

\end{proof}

\begin{proof}[\textbf{proof of lemma \ref{lem:normalposcycle} (ii)}]
\textit{Step 1}. Let us think of $P$ and $P_{\rho}:= {\pt}_*\left(T \res (\Sc \setminus \Scr)\right)$ as currents in the open set $\Ac$ in the manifold $\widetilde{\C}^{n+1}$. Given a sequence $\rho_k \to 0$, we want to observe the boundaries $\p P_{\rho_k}$. Up to a subsequence we may assume that $\rho_k$ is such that $T_{0, \rho_k} \rightharpoonup T_\infty$ for a certain cone. Then the boundaries $\p P_{\rho_k}$ satisfy, as  $k \to \infty$, by the definition (\ref{eq:dilblow}) of $\La_{\rho_k}$:

\be
\label{eq:boundariesdilblow}
(\Lambda_{\rho_k})_*(\p P_{\rho_k}) =  -{\pt}_* \langle T_{0, \rho_k}, |z|=1 \rangle \rightharpoonup -{\pt}_* \langle T_\infty, |z|=1 \rangle .
\ee

Recall that we are viewing $P_{\rho_k}$ as currents in the open set $\Ac$, so also $T \res (\Sc \setminus \Scr)$ should be thought of as a current in the open set $\Sc$: this is why the only boundary comes from the slice of $T$ with $|z|=\rho_k$.

\medskip

Moreover if the sequence is chosen (and we will do so) as in the 1st step of the proof of lemma \ref{lem:normalposcycle} (i), then $(\Lambda_{\rho_k})_*(\p P_{\rho_k})$ have equibounded masses, since so do the $\p (T_{0, \rho_k})$'s and $\P^{-1}$ is a diffeomorphism on $\p B_1$.

The current $T_\infty$ has a special form: it is a positive-$(p,p)$-cone, so the $(2p-1)$-current $\langle T_\infty, |z|=1 \rangle$ has an associated $(2p-1)$-vector field that always contains the direction tangent to the Hopf fibers\footnote{The Hopf fibration is defined by the projection $H: S^{2n+1} \subset \C^{n+1} \rightarrow \CP^n$, $H(z_0, ..., z_n) = [z_0, ..., z_n]$. The Hopf fibers $H^{-1}(p)$ for $p \in \CP^n$ are maximal circles in $S^{2n+1}$, namely the links of complex lines of $\C^{n+1}$ with the sphere.} of $S^{2n+1}$. 

\medskip

\textit{Step 2}. We want to show that $P$ is a cycle in $\Ac$, i.e. that $\p P_{\rho_k} \rightharpoonup 0$ as $n \to \infty$. The boundary in the limit could possibly appear on $\CP^n \times \{0\}$ and we can exclude that as follows.

Let $\alpha$ be a $(2p-1)$-form of comass one with compact support in $\Ac$ and let us prove that $\p P_{\rho_k}(\alpha) \to 0$. Since $\Ac$ is a submanifold in $\CP^n \times \C^{n+1}$, we can extend $\alpha$ to be a form in $\CP^n \times \C^{n+1}$. Let us write, using horizontal coordinates $\{t_j\}_{j=1}^{2n}$ on $\CP^n$ and vertical ones $\{s_j\}_{j=1}^{2n+2}$ for $\C^{n+1}$, 

$$\alpha=\alpha_h + \alpha_{v1} + \alpha_{v2} + ... \alpha_{v(2p-1)} ,$$

 where $\alpha_h$ is a form only in the $d t_j$'s, and each $\alpha_{vj}$ (for $j=1,2, ... , (2p-1)$) contains wedge products of $(2p-1-j)$ of the $d t_j$'s and $j$ of the $d s_j$'s. Rewrite, viewing $P_{\rho_n}$ as currents in $\CP^n \times \C^{n+1}$,

$$\p P_{\rho_k}(\alpha) = \left[(\Lambda_{\rho_k})_*(\p P_{\rho_k})\right]\left( \Lambda_{\rho_k}^{-1})^* \alpha \right).$$

The map $\Lambda_{\rho_k}^{-1}$ is expressed in our coordinates by $(t_1, ..., t_{2n}, s_1, ...s_n) \to (t_1, ..., t_{2n}, \rho_k s_1, ... \rho_k s_{2n+2})$, therefore 

$$(\Lambda_{\rho_k}^{-1})^* \alpha = \alpha^n_h + \alpha^n_{v1} + \alpha^n_{v2} + ... \alpha^n_{v(2p-1)} ,$$

where the decomposition is as above and with $\|\alpha^n_h\|^* \approx \|\alpha_h\|^*$ and $\|\alpha^n_{vj}\|^* \lesssim (\rho_k)^j \|\alpha_v\|^*$. The signs $\approx$ and $\lesssim$ mean respectively equality and inequality of the comasses up to a dimensional constant, so independently of the index $n$ of the sequence. 

\medskip

As $k \to \infty$ it holds $\alpha^k_h \to \alpha^\infty_h$ in some $C^\ell$-norm, where $\|\alpha^\infty_h\|^* \lesssim 1$ and $\alpha^\infty_h$ is a form in the $d t_j$'s \footnote{More precisely $\alpha^\infty_h$ coincides with the restriction of $\alpha_h$ to $\CP^n \times \{0\}$, extended to $\CP^n \times \C^{n+1}$ independently of the $s_j$ variables.}. We can write

$$\left|\left[(\Lambda_{\rho_k})_*(\p P_{\rho_k})\right]( \alpha^k_h)\right| \leq \left|\left[(\Lambda_{\rho_k})_*(\p P_{\rho_k})\right]( \alpha^k_h - \alpha^\infty_h)\right| + \left|\left[(\Lambda_{\rho_k})_*(\p P_{\rho_k})\right]( \alpha^\infty_h)\right| $$

and both terms on the r.h.s. go to $0$. The first, since $M((\Lambda_{\rho_k})_*(\p P_{\rho_k}))$ are equibounded and $|\alpha^k_h - \alpha^\infty_h|  \to 0$; the second because we can use (\ref{eq:boundariesdilblow}) and ${\pt}_* \p (T_\infty)$ has zero action on a form that only has the $d t_j$'s components, as remarked in step 1.

\medskip

Moreover 

$$\left|\left[(\Lambda_{\rho_k})_*(\p P_{\rho_k})\right]( \alpha^k_{vj})\right|  \to 0$$

for any $j \in \{1,2, ... (2p-1)\}$, because the currents $(\Lambda_{\rho_k})_*(\p P_{\rho_k}) =  -{\pt}_* \langle T_{0, \rho_k}, |z|=1 \rangle$ have equibounded masses by the choice of $\rho_k$, while the comasses $\|\alpha^k_{vj}\|^* \lesssim (\rho_k)^j \|\alpha_v\|^*$ go to $0$.

\medskip

Therefore no boundary appears in the limit and $P$ is an integral cycle in $\Ac$. The fact that it is semi-calibrated by $\Th$ follows easily by the fact that so are the currents $P_\rho$, as remarked just before lemma \ref{lem:normalposcycle}.

\end{proof}

\medskip

Summarizing, we are now able to take the proper transform of an integral cycle $T$ semi-calibrated by $\Om$ in $\Sce \subset B_1^{2n+2}$ and get an integral cycle $P$ in $\Ace$ that is semi-calibrated by $\Th$, where the semicalibration $\Th$ is Lipschitz (and actually smooth away from $\CP^n \times \{0\}$). Therefore the almost monotonicity formula holds true for $P$ in $\Ace$.

\section{Proof of the results}
\label{proof}

With the assumptions in proposition \ref{Prop:restate} we have to observe a sequence $(\la_{r_n})_* T$ as $r_n \to 0$. Recall that we have assumed (see (\ref{eq:psi})) that the ``straight'' $2$-planes $D^X$ are pseudo-holomorphic for $J$.

\medskip

Take any converging sequence $T_{0,r_n}:=(\la_{r_n})_* T \to T_\infty$ for $r_n \to 0$. Take the proper transform of each $T_{0,r_n}$ and denote it by $P_n$. Remark that $P_n$ is a $\Th_n$-semi-calibrated cycle, for a semicalibration $\Th_n$ that is smooth away from $\CP^n \times \{0\}$ and Lipschitz-continuous, with $|\Th_n-\Th_0| < c_n \text{dist}_{g_0}(\cdot, \CP^n \times \{0\})$ and the constants $c_n$ go to $0$ as $n \to \infty$ (lemma \ref{lem:lipcontrolI}).

From lemma \ref{lem:normalposcycle}, the masses of $P_n$ are uniformly bounded in $n$, since so are the masses of $T_{0,r_n}$ (by almost-monotonicity). 

So by compactness, up to a subsequence that we do not relabel, we can assume $P_n \rightharpoonup P_\infty$ for a normal cycle $P_\infty$. 

\medskip

\begin{lem}
 \label{lem:pinfty}
 $P_\infty$ is a $\Th_0$-semi-calibrated cycle; more precisely it is the proper transform of $T_\infty$.
\end{lem}

\begin{proof}
$\Th_0$-positiveness follows straight from the $\Th_n$-positiveness of $P_n$ and $|\Th_n-\Th_0| < c_n \text{dist}_{g_0}(\cdot, \CP^n \times \{0\})$, $c_n \to 0$.

The proper transform is a diffeomorphism away from the origin, thus

$$P_\infty \res (\Ac \setminus \Acr) = \lim_n {\pt}_* T_{0,r_n} \res (\Sce \setminus \Scer) = {\pt}_* T_\infty \res (\Sce \setminus \Scer),$$

so in order to conclude that $P_\infty$ is the proper transform of ${\pt}_* T_\infty$ we only need to show $P_\infty= \lim_{\rho \to 0} P_\infty \res (\Ac \setminus \Acr)$, i.e. that $M( P_\infty \res \Acr) \to 0$ as $\rho \to 0$.

Recall that $\vt_0=\vt_{\CP^n} + \vt_{\C^n}$; we want to estimate $M(P_\infty \res \Acr)= (P_\infty \res \Acr)(\Th_0) = \lim_{n \to \infty} (P_n \res \Acr)(\Th_0)$. Write

\be
\label{eq:est1}
(P_n \res \Acr)(\Th_0)=\frac{1}{p !}(P_n \res \Acr)((\vt_{\CP^n})^p) + \frac{1}{p !}(P_n \res \Acr)(\sum_{k=1}^{p} (\vt_{\CP^n})^{p-k} (\vt_{\C^n})^k).
\ee

The second term on the r.h.s. is bounded as follows:

$$(P_n \res \Acr)\left(\sum_{k=1}^{p} (\vt_{\CP^n})^{p-k} (\vt_{\C^n})^k\right) = (\La_\rho)_*(P_n \res \Acr) \left ((\La_r^{-1})^*\left( \sum_{k=1}^{p} (\vt_{\CP^n})^{p-k} (\vt_{\C^n})^k\right) \right).$$

The current $(\La_\rho)_*(P_n \res \Acr)$ is the proper transform of $T_{0, \rho r_n}$, therefore $M\left((\La_\rho)_*(P_n \res \Acr)\right) \leq K$ independently of $n$; the form in brackets has comass bounded by $\rho^2$. Altogether 

$$(P_n \res \Acr)\left(\sum_{k=1}^{p} (\vt_{\CP^n})^{p-k} (\vt_{\C^n})^k\right) \leq C(p,n) \rho^2 .$$

To bound the first term on the r.h.s. of (\ref{eq:est1}), let $P$ be the proper transform of $T$; using $(\La_r)^* (\vt_{\CP^n})^{p} = (\vt_{\CP^n})^{p}$ we can write

$$(P_n \res \Acr)((\vt_{\CP^n})^p)) = (P \res {\Ac}_{r_n \rho}) (\vt_{\CP^n})^{p} \leq M\left( P \res {\Ac}_{r_n \rho} \right) \leq M\left( P \res {\Ac}_{\rho} \right) $$

so it goes to $0$ as $\rho \to 0$ (uniformly in $n$). Summarizing we get that $(P_n \res \Acr)(\Th_0)$ is $o(1)$ of $\rho \to 0$ uniformly in $n$; therefore so is $M(P_\infty \res \Acr) = \lim_{n \to \infty} (P_n \res \Acr)(\Th_0)$ and we can write $P_\infty = \lim_{\rho \to 0} P_\infty \res (\Ac \setminus \Acr)$.

\medskip

\end{proof}

For a positive-$(p,p)$ integral cycle in $\R^{2n+2}$, as we have already mentioned, each tangent cone is determined by a $(p-1, p-1)$ integral cycle $L_\infty$ in $\CP^n$ that is calibrated by $\frac{(\theta_{\CP^n})^{p-1}}{(p-1)\text{!}}$, the normalized power of the K\"ahler form.

The previous lemma tells us that, for a sequence $r_n \to 0$ such that $T_{0,r_n} \rightharpoonup T_\infty$, the proper transforms $P_n := {\pt}_* T_{0,r_n}$ converge to ${\pt}_* T_\infty$, i.e. to the current $L_\infty \times \llbracket D^2 \rrbracket$. Indeed, the fact that a cone $T_\infty$ is radially invariant translates into the fact that its proper transform is invariant under the action of $\Lambda_\rho$ for any $\rho>0$.

\begin{proof}[\textbf{proof of proposition \ref{Prop:restate}}]
 
Let $T_{0,r_n} \rightharpoonup T_\infty$, a possible tangent cone. Let $L_\infty$ be the holomorphic $(p-1, p-1)$-integral cycle in $\CP^n$ that identifies $T_\infty$. 

If $y_0$ is a point in the support of $L_\infty$, then there exists a sequence of points $0 \neq y_j \to 0$ such that $H\left(\frac{y_j}{|y_j|}\right) \to y_0$ (where $H$ is the Hopf projection) and radii $\delta_j$ such that each ball $B_{\delta_j}(y_j)$ contains a set $\mathcal{C}_j$ of strictly positive $\mathcal{H}^{2p}$-measure, $\mathcal{C}_j$ is contained in the support of $T_{0,r_n}$ and the balls $B_{\delta_j}(y_j)$ are disjoint.

If we take a different sequence $T_{0,R_n} \rightharpoonup \tilde{T}_\infty$, we still find a sequence of points as before, since $\frac{y_j}{|y_j|}$ is not changed under radial dilations. Take the proper transforms $\tilde{P}_n$ of $T_{0,R_n}$. The density is preserved almost everywhere under the push-forward via a diffeomorphism. For each $\tilde{P}_n$ we find that, by upper semi-continuity of the density, $y_0$ is a point of density $\geq 1$ for $\tilde{P}_n$ (for all $n$). Therefore $y_0$ is of density $\geq 1$ for the limit of the currents $\tilde{P}_n$, i.e. for $\tilde{P}_\infty = {\pt}_* \tilde{T}_\infty$: this follows from the monotonicity formula, with an argument as in remark \ref{oss:contmass}.

This proves that any point in the support of $L_\infty$ is also in the support of $\tilde{L}_\infty$, the holomorphic $(p-1, p-1)$-integral cycle in $\CP^n$ that identifies $\tilde{T}_\infty$. Since $\tilde{T}_\infty$ is an arbitrary tangent cone this concludes the proposition and therefore lemma \ref{lem:uniqsupport} is proved, i.e. all tangent cones must have the same support.
\end{proof}

Now we have to make sure that any two links $L_\infty$ and $\tilde{L}_\infty$ (obtained by a blow up with different sequences $r_n$ and $R_n$) have multiplicities that agree a.e. 

\medskip

The following lemma should be known, but we recall it for sake of completeness.

\begin{lem}
\label{lem:closedconn}
Let $C$ be a semicalibrated cycle of dimension $m$ in $\R^n$. For any $x_0$ the set of tangent cones to $C$ at $x_0$ is a closed and connected subset (for the flat distance, which metrizes the weak*-topology on currents of equibounded mass and boundary mass, see \cite{G}).
\end{lem}

\begin{proof}[\textbf{proof of lemma \ref{lem:closedconn}}]
Let $\Upsilon$ be the set of all possible tangent cones at $x_0$. Given a sequence $\{T_k\}_{k=1}^\infty$ in $\Upsilon$ assume that $T_k \rightharpoonup T$. We want to show that $T \in \Upsilon$. 

The assumption $T_k \in \Upsilon$ means that there exists a sequence $r_n^k \to 0$ such that as $n \to \infty$ we have $C_{x_0, r_n^k} \rightharpoonup T_k$. With a diagonal argument we get $T \in \Upsilon$.

\medskip

Now, to prove connectedness, assume by contradiction that $\Upsilon = \Upsilon_1 \cup \Upsilon_2$, where $\Upsilon_1$ and $\Upsilon_2$ are closed and disjoint. Then there exist (in the space of currents) $A_1$ and $A_2$ open disjoint neighbourhoods respectively of $\Upsilon_1$ and $\Upsilon_2$. The family of currents $C_{x_o, r}$ ($r \geq 0$) is continuous and should therefore accumulate (as $r \to 0$) also somewhere outside $A_1$ and $A_2$, contradiction.
\end{proof}

\begin{lem}
\label{lem:contdensity}
Let $C_n$ and $C$ be integral cycles of dimension $k$ calibrated by a $k$-form $\om$. Assume that $C_n \rightharpoonup C$ and that the support $\mathcal{C}$ is the same for all $C_n$ and $C$ and it is compact. Let $\nu_n(x)$ denote the density at $x$ for $C_n$ and  $\nu(x)$ analogously the density at $x$ for $C$ (dealing with calibrated cycles, each $\nu_n$ or $\nu$ is well-defined everywhere).  

Then for every $x \in \Cc$ it holds $\nu(x)= \lim_{n \to \infty} \nu_n(x)$.
\end{lem}

\begin{proof}[\textbf{proof of lemma \ref{lem:contdensity}}] We will achieve the proof in three steps.

\medskip

\textbf{Claim (i)} for every $x \in \Cc$ it holds $\nu(x) \geq \limsup_{n \to \infty} \nu_n(x)$.

\medskip

This follows from the monotonicity formula. Indeed, let $B_r(x)$ be the ball around $x$ with radius $r$. By remark \ref{oss:contmass}, the weak convergence $C_n \rightharpoonup C$ yields $M(C_n \res B_r(x)) \to M(C \res B_r(x))$. By monotonicity we have $M(C_n \res B_r(x)) \geq \alpha_k \nu_n(x) r^k$, thus it must hold, for all $r>0$,

$$M(C \res B_r(x)) \geq \alpha_k (\limsup_{n \to \infty} \nu_n(x)) r^k .$$

Since $\nu(x) = \lim_{r \to 0} \frac{M(C \res B_r(x))}{\alpha_k r^k}$ we can conclude \textbf{claim (i)}.

\medskip

\textbf{Claim (ii)} There exists $L>0$ such that $\nu_n, \nu \leq L$ everywhere on $\Cc$.

\medskip

For each fixed $C_n$ (resp. $C$), the density $\nu_n$ (resp. $\nu$) is a bounded function: this follows from the facts that the mass is locally finite, the monotonicity formula holds and the density is upper semi-continuous. So, in order to prove claim (ii), assume by contradiction that here exist points $p_n \in \Cc$ such that $\nu_n(p_n) \uparrow + \infty$ as $n \to \infty$. Up to a subsequence that we do not relabel we can assume $p_n \to p$ for a point $p$ in $\Cc$. Choose a ball $B_R(p)$ and let $m>0$ be chosen so that $M(C \res B_R(p))=\alpha_k \cdot m \cdot R^2$. Choose $n_0$ large enough so that for all $n \geq n_0$ it holds (i) $\theta_n(p_n) \geq 3m$ and (ii) $|p_n- p| < \frac{R}{10}$. Then consider the balls $B_{\frac{9 R}{10}}(p_n)$: they are contained in $B_R(p)$. 

By the monotonicity formula applied at $p_n$ for $C_n$, we get $M(C_n \res B_{\frac{9 R}{10}}(p_n)) \geq \alpha_k (3m) \frac{9^2 R^2}{10^2}$ and therefore $M(C_n \res B_R(p)) \geq \alpha_k (3m) \frac{9^2 R^2}{10^2}$. By remark \ref{oss:contmass} we must have $M(C_n \res B_R(p)) \to M(C \res B_R(p))$, so we can write $M(C \res B_R(p)) \geq \alpha_k (3m) \frac{9^2 R^2}{10^2}$. Since $\frac{3 \cdot 9^2 }{10^2}>1$ we contradicted the assumption $M(C \res B_R(p))=\alpha_k \cdot m \cdot R^2$.
\medskip

\textbf{Claim (iii)} for every $x \in \Cc$ it holds $\nu(x) = \lim_{n \to \infty} \nu_n(x)$.

\medskip

It suffices to show, for an arbitrary $x$, that 

\be
\label{eq:dsfg}
\nu(x) = \limsup_{n \to \infty} \nu_n(x).
\ee

Once this is achieved, choose a subsequence $n_k$ such that $\liminf_{n \to \infty} \nu_n(x) = \lim_{n_k \to \infty} \nu_{n_k}(x)$ and apply (\ref{eq:dsfg}) to the sequence of currents $C_{n_k}$ to show that $\nu(x) =\lim_{n_k \to \infty} \nu_{n_k}(x)= \liminf_{n \to \infty} \nu_n(x)$.

\medskip

Again the main ingredient for (\ref{eq:dsfg}) is the monotonicity formula, which for an arbitrary $x \in \Cc$ states

\be
\label{eq:2mono}
\begin{split}
M(C_n \res B_R(x)) = \nu_n(x) + \int_{B_R(x)} \frac{|\vec{C}_y \wedge \p_r|^2}{|y-x|^k} \nu_n(y) d {\Hc}^k(y) \res \Cc \\
M(C \res B_R(x)) = \nu(x) + \int_{B_R(x)} \frac{|\vec{C}_y \wedge \p_r|^2}{|y-x|^k} \nu(y) d {\Hc}^k(y) \res \Cc,
\end{split}
\ee

where the unit simple $k$-vector $\vec{C}_y$ represents the approximate tangent to $\Cc$ at $y$ with the orientation given on $C_n$ and $\p_r$ is the radial unit vector (with respect to the point $x$). Therefore the function $\frac{|\vec{C}_y \wedge \p_r|^2}{|y-x|^k}$ is independent of $n$ (since the underlying $\Cc$ is always the same and $\pm \vec{C}_y$ both yield the same value for $|\vec{C}_y \wedge \p_r|$).

Let $\mu$ be the finite measure $\frac{|\vec{C}_y \wedge \p_r|^2}{|y-x|^k} \cdot {\Hc}^k(y) \res (\Cc \cap B_R(x))$. By \textbf{claim (ii)} we can apply Fatou's lemma to $L-\nu_n$ and $L-\nu$ to get 

$$\int \limsup_n \nu_n(y) d \mu(y) \geq \limsup_n \int  \nu_n(y) d \mu(y) ,$$

which together with \textbf{claim (i)} yields 

$$\int \nu(y)d \mu(y)  \geq \limsup_n \int  \nu_n(y) d \mu(y) .$$

We can now use this last inequality together with claim (i) and the fact that $M(C_n \res B_R(x)) \to M(C \res B_R(x))$ to pass to the limit in (\ref{eq:2mono}) as $n \to \infty$: we get that necessarily we have the equality $\nu(x) = \limsup_{n \to \infty} \nu_n(x)$. 

\end{proof}

\begin{proof}[\textbf{proof of theorem \ref{thm:mainrect}}]
Let $\Upsilon$ be the family of possible tangent cones to $T$ at $x_0$. The elements of $\Upsilon$ are integral $(p-1, p-1)$-cycles (in $\CP^{n}$, the projective space of $\C^{n+1} \equiv T_x \Mc$) calibrated by $\frac{(\theta_{\CP^n})^{p-1}}{(p-1)\text{!}}$ and by lemma \ref{lem:uniqsupport} they all have the same support.  

\medskip

First we are going to prove that there exists a subset ${\Upsilon}_d \subset \Upsilon$ that is countable and dense in $\Upsilon$, i.e. $\Upsilon$ is separable. This is achieved as follows.

All currents in $\Upsilon$ are supported on the same rectifiable set $\mathcal{C}$ and they can only differ by the choice of the density function. We can represent $\mathcal{C} \setminus \tilde{\mathcal{C}}$, where $\tilde{\mathcal{C}}$ is a ${\Hc}^{2p}$-null set, as the image of a Borel subset $K$ of $\R^{2p}$ via a Lipschitz map taking values in $\CP^{n}$ and with Lipschitz constant $\frac{1}{2} \leq L \leq 2$. To obtain this representation, recall (see \cite{Morgan}) that $\mathcal{C} \setminus \tilde{\mathcal{C}}$ is a countable union of disjoint pieces, each piece being the image, via a Lipschitz map with constant close to $1$, of a compact subset of $\R^{2p}$. We can freely change by translation the position of these countably many compact subset of $\R^{2p}$ and make them disjoint, so by denoting their union with $K$ we get the desired representation for $\mathcal{C} \setminus \tilde{\mathcal{C}}$.

For each current in $\Upsilon$ the density on the rectifiable set $\Cc$ is an $L^1$ function on $\left(\Cc, {\Hc}^{2p}\right)$. Using the previous representation of $\Cc$, we can record these densities as $L^1$ functions on $\R^{2p}$ that are zero outside of $K$. The family $\Upsilon$ therefore yields a family $\{g_a\}_{a \in \Upsilon}$ of such $L^1(\R^{2p})$ functions. Every such $L^1$ function on $\R^{2p}$ is associated to a current supported on ${\Cc}$. 

The family $\{g_a\}$ is compact in $L^1$: indeed the $L^1$-convergence for a sequence in $\{g_a\}$ yields the (weak*) convergence for the corresponding currents. So $\{g_a\}$ is closed with respect to the $L^1$-norm, because $\Upsilon$ is closed with respect to the weak*-topology. Moreover $\{g_a\}$ is bounded in $L^1$ because $\int_{\R^{2p}} g_a d \mathcal{L}^{2p}$ is comparable (up to a factor $2$, recall the condition on $L$) to the mass of the corresponding current which is fixed for all elements of $\Upsilon$. 

We conclude that, as a compact subspace of the separable normed space $L^1(\R^{2p})$, $\{g_a\}$ is also separable. The corresponding countable set of currents is the desired ${\Upsilon}_d$.

\medskip

Except on a ${\Hc}^{2p}$-null set $\tilde{\mathcal{C}'}$, all points of $\mathcal{C} \setminus \tilde{\mathcal{C}'}$ have integer densities for all currents in ${\Upsilon}_d$.

Let now $x \in \mathcal{C} \setminus \tilde{\mathcal{C}'}$ and observe the function $F$ from ${\Upsilon}_d$ to $\R$ assigning to every current $P \in {\Upsilon}_d$ the value $F(P):=\nu_P(x)$, where $\nu_P$ is the density of $P$. By lemma \ref{lem:contdensity} the function $F$ is continuous on the metric space  ${\Upsilon}_d$, but since it is also integer-valued it must be locally constant on ${\Upsilon}_d$.

\medskip

${\Upsilon}_d$ is dense in $\Upsilon$, so for avery current $P' \in \Upsilon$ the value $\nu_{P'}(x)$ is also locally constant by lemma \ref{lem:contdensity}. Since $\Upsilon$ is connected, $\nu_{P'}(x)$ must then be globally constant for $P' \in \Upsilon$. The point $x \in \mathcal{C} \setminus \tilde{\mathcal{C}'}$ was arbitrary, therefore all currents in $\Upsilon$ have a fixed density at all points except on the null set $\tilde{\mathcal{C}'}$ and this makes them equal as currents. A posteriori also the density on $\tilde{\mathcal{C}'}$ is fixed.

We have therefore obtained that $\Upsilon$ is made of one single element and we can conlcude the uniqueness theorem \ref{thm:mainrect} for tangent cones.
\end{proof}


\begin{thebibliography}{99}

\bibitem{Ale} Alexander, H. {\it Holomorphic chains and the support hypothesis conjecture}  J. Amer. Math. Soc.  10  (1997),  no. 1, 123--138.
\bibitem{BR} Bellettini, Costante and Rivi\`ere, Tristan {\it The regularity of Special Legendrian integral cycles}, to appear in Ann. Sc. Norm. Sup. Pisa Cl. Sci.
\bibitem{B}
Bellettini, Costante {\it Almost complex structures and calibrated integral cycles in contact 5-manifolds}, preprint 2010.
\bibitem{B2}
Bellettini, Costante {\it Tangent cones to positive-$(1,1)$ De Rham currents}, preprint 2011.
\bibitem{DG} De Giorgi, Ennio {\it 
Nuovi teoremi relativi alle misure $(r-1)$-dimensionali in uno spazio ad $r$ dimensioni (Italian)} Ricerche Mat. 4 (1955), 95-113.
\bibitem{DoT} S.K. Donaldson and R.P. Thomas {\it Gauge Theory in higher dimensions,} in
"The
geometric Universe", Oxford Univ. Press, 1998, 31-47.
\bibitem{DuzS} Duzaar, Frank and Steffen, Klaus {\it $\lambda$ minimizing currents}, Manuscripta Math. {80}, (1993), 4, 403-447.
\bibitem{F} Federer, Herbert {\it Geometric measure theory}, Die Grundlehren der mathematischen Wissenschaften, Band 153, Springer-Verlag New York Inc., New York, 1969, xiv+676.
\bibitem{G} M. Giaquinta, G. Modica and J. Sou{\v{c}}ek {\it Cartesian currents in the calculus of variations {I}}, Ergeb. Math. Grenzgeb. (3) vol.  37, Springer-Verlag, Berlin, 1998, xxiv+711.
\bibitem{HL} Harvey, Reese and Lawson, H. Blaine Jr. {\it Calibrated geometries}, {Acta Math.},{148}, {47--157},{1982}.
\bibitem{HS} Harvey, Reese; Shiffman, Bernard {\it A characterization of holomorphic chains}  Ann. of Math. (2)  99  (1974), 553--587.
\bibitem{K} King, James R. {\it The currents defined by analytic varieties}  Acta Math.  127  (1971), no. 3-4, 185--220.
\bibitem{Kis} Kiselman, Christer O.
{\it Tangents of plurisubharmonic functions} International Symposium in Memory of Hua Loo Keng, Vol. II (Beijing, 1988), 157-167, Springer, Berlin
\bibitem{MS} McDuff, Dusa and Salamon Dietmar {\it Introduction to symplectic topology}, {Oxford Mathematical Monographs}, {2}, {The Clarendon Press Oxford University Press}, {New York}, {1998}, {x+486}.
\bibitem{Morgan} Morgan, Frank, {\it Geometric measure theory}, Fourth edition, A beginner's guide, Elsevier/Academic Press, Amsterdam, 2009, viii+249.
\bibitem{PR} Pumberger, David and Rivi{\`e}re, Tristan  {\it Uniqueness of tangent cones for semi-calibrated 2-cycles}, Duke Math. J., Duke Mathematical Journal, 152  (2010),  no. 3, 441--480.
\bibitem{RT1} 
Rivi\`ere, Tristan; Tian, Gang {\it The singular set of $J$-holomorphic maps into projective algebraic varieties}, 
 J. Reine Angew. Math.  570  (2004), 47--87. 58J45
\bibitem{RT2} Rivi{\`e}re, Tristan and Tian, Gang  {\it The singular set of 1-1 integral currents}, Ann. of Math. (2), Annals of Mathematics. Second Series, 169, 2009, 3, 741-794.
\bibitem{SimonNotes} Simon, Leon {\it Lectures on geometric measure theory}, {Proceedings of the Centre for Mathematical Analysis, Australian National University},  {3}, {Australian National University Centre for Mathematical
              Analysis}, {Canberra}, {1983}, {vii+272}.
\bibitem{Simon} Simon, Leon {\it Asymptotics for a class of nonlinear evolution equations, with applications to geometric problems}, {Ann. of Math. (2)}, {Annals of Mathematics. Second Series}, {118} (1983), {3}, {525-571}.
\bibitem{SYZ} Strominger, Andrew and Yau, Shing-Tung and Zaslow, Eric \textit{Mirror symmetry is T-duality}, Nuclear
Phys. B 479 (1996), 243-259.
\bibitem{Ta} Taubes, Clifford Henry, "$\rm SW \Rightarrow Gr$: from the Seiberg-Witten equations to pseudo-holomorphic curves". {\it Seiberg Witten and Gromov invariants for symplectic 4-manifolds. },  1--102, First Int. Press Lect. Ser., 2, Int. Press, Somerville, MA, 2000.
\bibitem{Ti} G. Tian  {\it Gauge theory and calibrated geometry. {I}}, Ann. of Math. (2), 151 (2000) 1, 193--268.
\bibitem{Ti2} Tian, Gang {\it Elliptic {Y}ang-{M}ills equation}, {Proc. Natl. Acad. Sci. USA}, {Proceedings of the National Academy of Sciences of the United States of America}, {99}, (2002), {24}, 15281-15286.
\bibitem{W} White, Brian {\it Tangent cones to two-dimensional area-minimizing integral
              currents are unique}, Duke Math. J., Duke Mathematical Journal, 50, 1983, 1, 143--160.
\bibitem{Wirtinger}  Wirtinger, W.  {\it Eine {D}eterminantenidentit\"at und ihre {A}nwendung auf
              analytische {G}ebilde in Euklidischer und {H}ermitescher
              {M}assbestimmung}, Monatsh. Math. Phys. {44}, 1936 {1},  {343-365}.

\end{thebibliography}
\end{document}